\documentclass[11pt]{article}

\usepackage{amsmath}
\usepackage{amsfonts}
\usepackage{amsthm}
\usepackage{amssymb}
\usepackage{MnSymbol}
\usepackage{bbm}
\usepackage{listings}
\usepackage{enumerate}
\usepackage{cite}

\newcommand{\R}{\mathbb{R}}

\newcommand{\E}{\mathbb{E}}

\newcommand{\inr}[2]{\langle #1, #2 \rangle}

\newcommand{\pdiff}[2]{\frac{\partial #1}{\partial #2}}
\newcommand{\pdiffII}[3]{\ifthenelse{\equal{#2}{#3}}
{\frac{\partial^2 #1}{\partial #2^2}}
{\frac{\partial^2 #1}{\partial #2 \partial #3}}
}
\newcommand{\diffII}[3]{\ifthenelse{\equal{#2}{#3}}
{\frac{d^2 #1}{d #2^2}}
{\frac{d^2 #1}{d #2 d #3}}
}
\newcommand{\diff}[2]{\frac{d #1}{d #2}}

\newcommand{\grad}{\nabla}

\renewcommand{\Pr}{\text{Pr}}

\newtheorem{theorem}{Theorem}[section]

\newtheorem{lemma}[theorem]{Lemma}
\newtheorem{corollary}[theorem]{Corollary}
\newtheorem{proposition}[theorem]{Proposition}

\newtheoremstyle{example}{\topsep}{\topsep}%
     {}
     {}
     {\bfseries}
     {}
     {\newline}
     {\thmname{#1}\thmnumber{ #2}\thmnote{ #3}}

\theoremstyle{example}

\newcommand{\I}{\mathcal{I}}

\title{A multidimensional version of noise stability}
\author{Joe Neeman}

\begin{document}
\maketitle
\begin{abstract}
 We give a multivariate generalization of Borell's noise stability
 theorem for Gaussian vectors.
 As a consequence we recover two inequalities,
 also due to Borell, for exit times of the Ornstein-Uhlenbeck process.
\end{abstract}

\section{Introduction}
There has been a recent flurry of activity in
probability~\cite{MoOdOl:10,Mossel:10} and computer science~\cite{KKMO:04,
KindlerOdonnell:12,Raghavendra:08}
around a certain paper of Borell~\cite{Borell:85} on inequalities
satisfied by the Ornstein-Uhlenbeck process.
In his paper, Borell proved a theorem, which is somewhat complicated to
state, showing that certain quantities only decrease under
Ehrhard symmetrization~\cite{Ehrhard:84}. He then derived two simpler
corollaries, about hitting times for the Ornstein-Uhlenbeck process,
from this general result.

We recall that the Ornstein-Uhlenbeck process on $\R^n$ is the Gaussian process
$\{X_t: t \in \R\}$ with mean zero and covariance $\E X_s X_t^T = e^{|t-s|} I_n$.
This is a Markov process, as may be seen by the construction
$X_t = e^{-t} B_{e^{2t}}$ for a Brownian motion $B_t$, and the stationary
measure of $X_t$ is the standard Gaussian measure, $\gamma_n$.
For a set $A \subset \R^n$, we denote its exit time under $X_t$ by
$e_A = \inf \{t \ge 0: X_t \not \in A\}$.
Although they were originally written in terms of hitting times instead of
exit times, Borell's two corollaries of his general inequality may be
written as follows, in which \emph{half-space} means a set of the form
$\{x \in \R^n : x\cdot a \le b\}$, and half-spaces are \emph{parallel} if they
have the same normal vector.
\begin{theorem}[Borell]\label{thm:exit-time}
If $B \subset \R^n$ is a half-space with $\gamma_n(B) = \gamma_n(A)$
then $e_B$ stochastically dominates $e_A$; i.e., for every $t \ge 0$,
\[
 \Pr(e_A \ge t) \le \Pr(e_B \ge t).
\]
\end{theorem}

\begin{theorem}[Borell]\label{thm:occupation-time}
If $B_1$ and $B_2$ are parallel half-spaces with 
$\gamma_n(B_i) = \gamma_n(A_i)$ then
\[
 \E \int_0^{t \land e_{A_1}} 1_{A_2}(X_s)\ ds
 \le \E \int_0^{t \land e_{B_1}} 1_{B_2}(X_s)\ ds.
\]
\end{theorem}

There is a third corollary of Borell's general inequality that did not appear
in his original paper~\cite{Borell:85}, but has nevertheless become widely
applied in theoretical computer science, particularly in the study
of hardness of approximation.

\begin{theorem}[Borell]\label{thm:noise-sense}
If $B_1$ and $B_2$ are parallel half-spaces with 
$\gamma_n(B_i) = \gamma_n(A_i)$ then for any $t > 0$,
\[
 \Pr(X_0 \in A_1, X_t \in A_2) \le \Pr(X_0 \in B_1, X_t \in B_2).
\]
\end{theorem}

In the special case $A_1 = A_2$, this inequality is sometimes interpreted
as showing that half-spaces are the most ``noise stable'' sets. Here,
we think of $X_t$ as being a noisy version of $X_0$, and so a set $A$ is
noise stable if the event $\{X_0 \in A\}$ tends to agree with the
event $\{X_t \in A\}$. Theorem~\ref{thm:noise-sense} implies that this
overlap is maximized, over all sets with a fixed Gaussian measure, by
half-spaces. Using an invariance principle, Mossel et al.~\cite{MoOdOl:10}
deduced from Theorem~\ref{thm:noise-sense} a similar inequality on
the discrete cube (although the statement on the cube is necessarily more
complicated, because the direction of a half-space's normal vector becomes
important); that work then laid the foundation for many applications in
theoretical computer science (for a few examples, see~\cite{KKMO:04,KindlerOdonnell:12,Raghavendra:08}).

Note that in Theorem~\ref{thm:noise-sense}, the joint distribution of
$(X_0, X_t)$ has mean zero and covariance
$\big(\begin{smallmatrix} 1 & \rho \\ \rho & 1\end{smallmatrix}\big) \otimes I_n$,
where $\rho = e^{-t}$.
Our main result is a multivariate generalization of Theorem~\ref{thm:noise-sense},
which allows for more than two Gaussian vectors and a more general
covariance structure than that endowed by the
Ornstein-Uhlenbeck process. Our generalization is also strong enough to recover
Theorems~\ref{thm:exit-time} and~\ref{thm:occupation-time}; we are thankful
to Michel Ledoux for pointing this out.

 \begin{theorem}\label{thm:main}
  Let $M = (m_{ij})$
  be a $k \times k$ positive semidefinite matrix with $m_{ij} \ge 0$,
  and let $X = (X_1, \dots, X_k)$ be a $kn$-dimensional Gaussian vector
  with covariance $M$.
  For any measurable $A_1, \dots, A_k \subset \R^n$,
  \begin{equation}\label{eq:main}
   \Pr(X_i \in A_i \text{ for all } i)
   \le \Pr(X_i \in B_i \text{ for all } i)
  \end{equation}
  whenever $B_i$ is a collection of parallel half-spaces with
  $\gamma_n(B_i) = \gamma_n(A_i)$.
  
  Moreover, if equality is attained in~\eqref{eq:main} then there exists
  a collection $(B_1, \dots, B_k)$ of parallel half-spaces such that
  $A_i = B_i$ up to sets of measure zero.
 \end{theorem}

 By setting $k = 2$, 
 Theorem~\ref{thm:main} recovers Theorem~\ref{thm:noise-sense}.
 We should remark that a generalization along these lines,
 but nevertheless different from Theorem~\ref{thm:main}, has already
 been discovered: Isaksson and Mossel~\cite{IsakssonMossel:12} showed that
 the inequality~\eqref{eq:main} also holds
 under the hypothesis that the off-diagonal elements of $M^{-1}$
 are non-positive. In other words, our hypothesis is that every pair
 $X_i, X_j$ is positively correlated, while~\cite{IsakssonMossel:12} assumed
 that every pair $X_i, X_j$ is conditionally positively correlated given
 all the other variables. Neither of these conditions is strictly stronger than
 the other, and in fact both would suffice for recovering
 Theorems~\ref{thm:exit-time} and~\ref{thm:occupation-time}.
 However,~\cite{IsakssonMossel:12} did not characterize the equality cases
 of~\eqref{eq:main} under their hypothesis.
 
 Next, we will show how Theorem~\ref{thm:main} may be used to recover
 Theorem~\ref{thm:exit-time}. This reduction is quite similar to one by
 Burchard and Schmuckenschlager~\cite{BurchardSchmuckenschlager:01}, who
 were studying exit times of Brownian motion on manifolds. (In that case,
 the study of exit times has a fairly long history;
 see~\cite{BurchardSchmuckenschlager:01} for references.) As we will see, though,
 our approach to Theorem~\ref{thm:main} is quite different to that
 of Burchard and Schmuckenschlager, who studied two-point symmetrizations.
 
 Let $X_t$ be the Ornstein-Uhlenbeck process,
 and consider the finite dimensional marginal
 $(X_{t_1}, \dots, X_{t_k})$ for a sequence of times $t_1 < \cdots < t_k$.
 This is a mean-zero Gaussian vector with covariance $M \otimes I_n$, where
 $m_{ij} = e^{-|t_i - t_j|}$. Clearly, then, $M$ satisfies the hypothesis
 of Theorem~\ref{thm:main} and in particular, we have
 \begin{equation}\label{eq:discrete-hitting-time}
  \Pr(X_{t_i} \in A \text{ for all } i) \le \Pr(X_{t_i} \in B \text{ for all } i)
 \end{equation}
 when $B$ is a half-space with $\gamma_n(A) = \gamma_n(B)$.
 This is essentially a discrete version of Theorem~\ref{thm:exit-time},
 since the event $\{X_{t_i} \in A \text{ for all } i\}$ is a discretization
 of $\{e_A \ge t_k\}$.
 
 To complete the proof, we need to show that one can
 take limits.
 Setting $t_i = it/k$ in~\eqref{eq:discrete-hitting-time}, we have
 \begin{align*}
 \Pr(e_A \ge \tau)
 &= \Pr(X_t \in A \text{ for all } 0 < t < \tau) \\
 &\le \Pr(X_{i\tau/k} \in A \text{ for all } i = 1, \dots, k) \\
 &\le \Pr(X_{i\tau/k} \in B \text{ for all } i = 1, \dots, k),
 \end{align*}
 where the last inequality follows from~\eqref{eq:discrete-hitting-time}.
 Next, we send $k \to \infty$. Recall that $X_t$ is uniformly continuous
 with probability 1. Hence for any $\epsilon > 0$, we make take $k$ large enough
 so that $|s - t| \le 1/k$ implies $|X_s - X_t| \le \epsilon$; for this $k$,
 \[
 \Pr(X_{i\tau/k} \in B \text{ for all } i = 1, \dots, k) \le
 \Pr(X_t \in B_\epsilon \text{ for all } 0 < t < \tau)
 = \Pr(e_{B_{\epsilon}} \ge \tau),
 \]
 where $B_\epsilon$ is the $\epsilon$-enlargement of $B$:
 $B_\epsilon = \{x \in R^n: d(x, B) \le \epsilon\}$. We have shown,
 therefore, that for any $\epsilon > 0$,
 $\Pr(e_A \ge \tau) \le \Pr(e_{B_\epsilon} \ge \tau)$. It only remains to
 show, then, that $\Pr(e_{B_\epsilon} \ge \tau)$ converges to
 $\Pr(e_B \ge \tau)$ as $\epsilon \to 0$.
 
 Consider instead the equivalent statement that $\Pr(e_{B_\epsilon} < \tau)$ converges
 to $\Pr(e_B < \tau)$.
 Since $B$ is closed and $X_t$ has continuous paths,
 $e_B < \tau$ implies that there is some $\epsilon > 0$ with
 $e_{B_\epsilon} < \tau$. That is, the function $1_{\{e_{B_\epsilon} < \tau\}}$
 converges pointwise (and upwards) to $1_{\{e_B < \tau\}}$ as $\epsilon \to 0$.
 By the monotone convergence theorem, it follows that $\Pr(e_{B_\epsilon} < \tau)$
 converges to $\Pr(e_B < \tau)$ as $\epsilon \to 0$. Hence
 \[
  \Pr(e_A \ge \tau) \le \lim_{\epsilon \to 0} \Pr(e_{B_\epsilon} \ge \tau)
  = \Pr(e_B \ge \tau)
 \]
 and so we have recovered Theorem~\ref{thm:exit-time}.
 
 We have mentioned already that it is also possible to recover
 Theorem~\ref{thm:exit-time}
 from the result in~\cite{IsakssonMossel:12}. Indeed, the matrix $M$ with
 entries $m_{ij} = e^{-|t_i - t_j|}$ does satisfy the hypothesis
 in~\cite{IsakssonMossel:12} (namely that the off-diagonal entries of its
 inverse are non-positive), although this is certainly less obvious
 then the fact that $M$ satisfies the conditions of Theorem~\ref{thm:main}.
 
 Let us also indicate how Theorem~\ref{thm:occupation-time} is recovered.
 We want to show that
 \begin{equation}\label{eq:occupation}
  \E \int_0^{t\land e_{A_1}} 1_{A_2}(X_s)\ ds
 \end{equation}
 is only increased when $A$ is replaced by $B$ (recall that $B_1$ and $B_2$ are
 parallel half-spaces satisfying $\gamma_n(B_i) = \gamma_n(A_i)$.
 We may suppose that $A_2 \subset A_1$, since if not then~\eqref{eq:occupation}
 may be trivially made larger by moving some of $A_2$'s mass inside $A_1$;
 if this is impossible because $\gamma_n(A_2) > \gamma_n(A_1)$
 then~\eqref{eq:occupation} is trivially bounded by $t \land e_{A_1}$,
 which, by Theorem~\ref{thm:exit-time}, is stochastically dominated
 by $t \land e_{B_1}$, and this in turn is equal to the right hand side
 of~\eqref{eq:occupation} with $B$ replacing $A$.
 
 Now that we have reduced to the case $A_2 \subset A_1$, we may
 discretize~\eqref{eq:occupation} as
 \[
  \frac{1}{k} \sum_{i=1}^k \Pr(X_{t_j} \in A_1 \text{ for all } j < i \text{ and } X_{t_i} \in A_2).
 \]
 By~\eqref{eq:discrete-hitting-time}, this is only increased when $A$ is
 replaced by $B$. To recover Theorem~\ref{thm:occupation-time} from here,
 it suffices to take a limit in much the same manner as before; we omit
 the details.

 \section{The Ornstein-Uhlenbeck semigroup}
 
 We will prove Theorem~\ref{thm:main} by differentiating a particular
 functional under the Ornstein-Uhlenbeck semigroup.
 This proof method has a long history, beginning with Varopoulos'
 work~\cite{Varopoulos:85} connecting the heat semigroup with Sobolev
 inequalties. More recently, and more apropos of this work, Bakry
 and Ledoux~\cite{BakryLedoux:96} proved the Gaussian isoperimetric inequality
 by differentiating Bobkov's functional~\cite{Bobkov:97} under the
 Ornstein-Uhlenbeck semigroup. We will follow quite a similar approach here,
 using a generalization of a functional that was introduced by Mossel
 and the author~\cite{MosselNeeman:12b} to prove Theorem~\ref{thm:main} in the case
 $k=2$.
 
 We define the Ornstein-Uhlenbeck semigroup $\{P_t: t \ge 0\}$, which
 acts on functions $f: \R^n \to \R^n$, by
 \[
  (P_t f)(x) = \int_{\R^n} f(e^{-t} x + \sqrt{1 - e^{-2t}}y)\ d\gamma_n(y),
 \]
 where $\gamma_n$ is the standard Gaussian measure on $\R^n$.
 Equivalently, $(P_t f)(X_0) = \E (f(X_t) \mid X_0)$, where $\{X_t: t \ge 0\}$
 is the Ornstein-Uhlenbeck process from the previous section. From either
 definition, one can easily see that $P_0$ is the identity operator, while
 $P_t f \to \E f$ as $t \to \infty$.

 One remarkable property of the Ornstein-Uhlenbeck semigroup is that
 it has very nice formulas for its commutation with smooth functions.
 In particular, for any $F = (f_1, \dots, f_k): \R^n \to \R^k$
 and for any smooth $\Psi: \R^k \to \R$,
 there is the formula (see, eg., ~\cite{Ledoux:00})
 \begin{equation}\label{eq:diffusion}
  \diff{}{s} P_s \Psi(P_{t-s} F)
  = \sum_{i,j = 1}^k \partial^2_{ij} \Psi(P_{t-s} F) \inr{\grad P_{t-s} f_i}
  {\grad P_{t-s} f_j}.
 \end{equation}

 We begin with an observation that comes, essentially, from
 applying~\eqref{eq:diffusion} to the function that is $\Psi$ composed
 with a linear operator. In the following, $\odot$ denotes the
 Hadamard (elementwise) product between two matrices and $H_\Psi$ denotes
 the Hessian matrix of $\Psi$.
 \begin{proposition}\label{prop:linear-diffusion}
  Let $M$ be a $k \times k$ positive-definite matrix with $m_{ii} = 1$ and let
  $X = (X_1, \dots, X_k)$ be a $kn$-dimensional Gaussian vector with
  mean zero and covariance $M \otimes I_n$. If $\Psi: [0, 1]^k \to \R$ satisfies
  $M \odot H_\Psi \le 0$ then for all measurable
  $F = (f_1, \dots, f_k): \R^n \to [0, 1]^k$,
  \[
   \E \Psi(F(X)) \le \Psi(\E F).
  \]
 \end{proposition}
 
 We remark that the assumption $m_{ii} = 1$ in
 Proposition~\ref{prop:linear-diffusion} is not necessary, but it makes our
 notation simpler since otherwise we need to consider Ornstein-Uhlenbeck semigroups
 with different stationary measures.
 
 Before proving Proposition~\ref{prop:linear-diffusion}, we introduce
 some notation that will be useful in what follows:
 for any $f: \R^n \to \R$
and any $n \times m$ matrix $M$, denote the function $f \circ M : \R^m \to \R$
by $f^M$.

\begin{proof}
Let $Q = (q_{ij})$ be the positive definite square root of $M \otimes I_n$,
and for $i = 1, \dots, k$, let $Q_i$ be the $n \times kn$ matrix
consisting of rows $(i-1)n + 1$ through $in$ of $Q$.
Let $Z$ be a standard Gaussian vector in $\R^{kn}$, and 
note that
$Q Z = (Q_1 Z, \dots, Q_k Z)$ is a $kn$-dimensional Gaussian vector with
mean 0 and covariance $M \otimes I_n$ (i.e., $QZ$ has the same distribution as $X$).
We consider the quantity
\[
 F(s, t, z) = \left(P_s \Psi(P_{t-s} f_1^{Q_1}, \dots, P_{t-s} f_k^{Q_k})\right)(z)
\]
for $s, t \in [0, \infty)$ and $z \in \R^{kn}$.
First, let us check how $P_t$ commutes with linear transformations.
Since $Q \otimes I_n$ is the square root of $M \otimes I_n$,
we have $Q_i^T Q_i = I_n$ and so
\begin{align}
 (P_t f_i^{Q_i})(x)
 &= \int_{\R^n} f_i(e^{-t} Q_i x + \sqrt{1-e^{-2t}} Q_i y)\: d\gamma_{kn}(y) \notag \\
 &= \int_{\R^n} f_i(e^{-t} Q_i x + \sqrt{(1-e^{-2t})} y)\: d\gamma_n(y) \notag \\
 &= (P_t f_i)^{Q_i}(x).
 \label{eq:Pt-linear}
\end{align}
Of course, the gradient commutes with linear transformations as
$\grad f^A = A^T (\grad f)^A$.
Combining this with~\eqref{eq:Pt-linear},
\[
\grad P_{t-s} f_i^{Q_i} = \grad (P_{t-s} f_i)^{Q_i}
= Q_i^T (\grad P_{t-s} f_i)^{Q_i}.
\]
In particular,
\begin{align}
\inr{\grad P_{t-s} f_i^{Q_i}}{\grad P_{t-s} f_j^{Q_j}}
&= \inr{Q_i^T (\grad P_{t-s} f_i)^{Q_i}}
{Q_j^T (\grad P_{t-s} f_j)^{Q_i}} \notag \\
&= m_{ij} \inr{(\grad P_{t-s} f_i)^{Q_i}}{(\grad P_{t-s} f_j)^{Q_i}}.
\label{eq:grad-linear}
\end{align}
For brevity, let $v_i = (\grad P_{t-s} f_i)^{Q_i}$.
Then, by~\eqref{eq:diffusion} and~\eqref{eq:grad-linear},
\begin{align}
 \pdiff{F(s,t,\cdot)}{s}
 &= P_s \sum_{i,j=1}^k
 \inr{\grad P_{t-s} f_i^{Q_i}}{\grad P_{t-s} f_j^{Q_j}}
 \partial^2_{ij} \Psi(P_{t-s} F^Q) \notag \\
 &= P_s \sum_{i,j=1}^k
 m_{ij} \inr{v_i}{v_j}
 \partial^2_{ij} \Psi(P_{t-s} F^Q).
\label{eq:exchangeable-derivative}
\end{align}
Note that if $v^T = (v_1^T \dots v_k^T)$ then the last line may be rewritten as 
\begin{equation}\label{eq:exchangeable-derivative-2}
P_s  \left(v^T ((M \odot H_\Psi) \otimes I_n) v\right).
\end{equation}
In particular, if $M \odot H_\Psi \le 0$ then
 $\pdiff{F(s,t,z)}{s} \le 0$ for every $s, t$ and $z$. Hence,
 $\lim_{t \to \infty} F(t,t,Z) \le \lim_{t\to\infty} F(0,t,Z)$. But since
 $(Q_1 Z, \dots, Q_k Z)$ has the same distribution as $(X^{(1)}, \dots, X^{(k)})$,
 $\E F(t,t,Z)$ converges to $\E \Psi(F(X))$
 and $\E F(0,t,Z)$ converges to $\Psi(\E F(X))$.
 \end{proof}
 
 With hardly any extra effort, the proof of Proposition~\ref{prop:linear-diffusion}
 also allows us to characterize its equality cases. Indeed, if
 $\E \Psi(F) = \Psi(\E F)$ then we must have $\pdiff{F(s, t, z)}{s} = 0$ for
 every $s$, $t$, and $z$. Going back to~\eqref{eq:exchangeable-derivative}
 and~\eqref{eq:exchangeable-derivative-2}, we see that
 $P_s \left(v^T ((M \odot H_\Psi) \otimes I_n) v\right)$ must be
 identically zero, and hence
 $((M \odot H_\Psi) \otimes I_n) v = 0$. In other words, we have
 the following corollary:
 \begin{corollary}\label{cor:linear-diffusion}
  Under the hypothesis of Proposition~\ref{prop:linear-diffusion}, if
  $\E \Psi(F) = \Psi(\E F)$ then for every $t > 0$,
  \[
   \big((M \odot H_\Psi(P_t F^Q)) \otimes I_n\big)
   \begin{pmatrix}
    (\grad P_t f_1)^{Q_1}\\
    \vdots \\
    (\grad P_t f_k)^{Q_k}
   \end{pmatrix}
   = 0 \text{ pointwise.}
  \]
 \end{corollary}

 \section{Proof of Theorem~\ref{thm:main}}
 
 Before proving Theorem~\ref{thm:main}, note that by translating $X$,
 $A_i$, and $B_i$, it suffices to consider the
 case in which $X$ has mean zero.
 Moreover, by scaling each $X_i$, $A_i$, and $B_i$, we may assume
 that $m_{ii} = 1$ for each $i$ (here and in the previous sentence
 we are using the fact that a collection of parallel half-spaces
 remains one under translation and scaling).
 We may also assume that $M$ is strictly positive definite, since if not then
 the distribution of $X$ is supported on a subspace, onto which we may project.
 
 Consider the function
 \begin{equation}\label{eq:def-J}
  J(x_1, \dots, x_k; M)
  = \Pr\left(X_i \le \Phi^{-1}(x_i) \text{ for all } i\right).
 \end{equation}
 Note that $\Pr(X_i \le \Phi^{-1}(x_i)) = x_i$; in particular, the collection
 $(B_1, \dots, B_k)$ with $B_i = \{y \in \R^n: y_1 \le x_i\}$ is a set of
 parallel half-spaces with $\Pr(X_i \in B_i) = x_i$. Since every such set of
 parallel half-spaces may be obtained by applying a fixed rotation to each
 $B_i$, Theorem~\ref{thm:main} is equivalent to the statement
 \[
  \Pr(X_i \in A_i \text{ for all } i) \le J(\gamma_n(A_1), \dots, \gamma_n(A_k); M).
 \]
 Next, note that if $x_1, \dots, x_k \in \{0, 1\}$ then
 $J(x_1, \dots, x_k)$ is 1 if all the $x_i$ are 1, and 0 otherwise. In
 particular,
 \[
  \Pr(X_i \in A_i \text{ for all } i) = \E J(1_{A_1}(X_1), \dots, 1_{A_k}(X_k); M)
 \]
 and hence~\eqref{eq:main} is equivalent to the statement
 \begin{equation}\label{eq:functional}
  \E J(f_1(X_1), \dots, f_k(X_k); M) \le J(\E f_1, \dots, \E f_k; M)
 \end{equation}
 in the special case $f_i = 1_{A_i}$. In fact, we will prove~\eqref{eq:functional}
 for general measurable functions $f_i: \R^n \to [0, 1]$.

 Unsurprisingly, the proof of~\eqref{eq:functional} goes through
 Proposition~\ref{prop:linear-diffusion}. The main task left, therefore,
 is to compute the Hessian of $J$ and check that it satisfies the hypothesis
 of Proposition~\ref{prop:linear-diffusion}.

 \begin{proposition}\label{prop:hess-J}
  If $m_{ii} = 1$ and $m_{ij} \ge 0$
  then $J$ satisfies $M \odot H_J \le 0$.
 \end{proposition}

 For a vector $v$, let $v_{\hat i}$ denote $v$ without the $i$th entry,
 and for a square matrix $M$, let $M_{\hat i}$ denote $M$ without the $i$th row
 and column.
 For a square matrix $M$, let $\overline{M_i}$ denote the Schur complement of
 $M_{\hat i}$ in $M$. In other words, $\overline{M_i}$ is defined by
 \[\overline{M_i}^{-1} = (M^{-1})_{\hat i}.\]
 A well-known formula for conditional distributions
 of Gaussian vectors~\cite{Eaton:83} states if $X$ has mean zero and covariance
 $M$ satisfying $m_{ii} = 1$,
 then conditioned on $X_i = x_i$, $X_{\hat i}$ has mean $x_i M_{i \hat i}$
 (where $M_i$ denotes the $i$th row of $M$, so $M_{i\hat i}$ is the $i$th
 row of $M$ with its $i$th element removed) and covariance
 $\overline{M_i}$.
 
 To compute the first derivatives of $J$, let
 \[
  K(x_1, \dots, x_k; M) = \Pr\left(X_i \le x_i \text{ for all } i\right),
 \]
 and note that
 $J(x_1, \dots, x_k) = K(\Phi^{-1}(x_1), \dots, \Phi^{-1}(x_k))$.
  Now, for any $i$ we may write $K$ as
 \[
  K(x; M) = \int_0^{x_i} \phi(y)
  \Pr(X_j \le x_j \text{ for all } j \ne i \mid X_i = y)\ dy,
 \]
 from which we see that
 \[
  \partial_i K(x; M)
  = \phi(x_i) \Pr(X_j \le x_j \text{ for all } j \ne i \mid X_i = x_i).
 \]
 Now, given that $X_i = x_i$, $X_{\hat i}$ has mean $M_{i\hat i} x_i$
 and covariance $\overline{M_i}$; hence,
 \[
  \Pr(X_j \le x_j \text{ for all } j \ne i \mid X_i = x_i)
  = K(x_{\hat i} - M_{i\hat i} x_i; \overline{M_i}),
 \]
 and so we have the formula
 \[
  \partial_i K(x; M) = \phi(x_i) K(x_{\hat i} - M_{i\hat i} x_i; \overline{M_i})
 \]
 (bear in mind that this formula is only valid under the assumption $m_{ii} = 1$;
 if not then $m_{ii}$ makes an appearance in the formula also).
 Applying the chain rule and the identity
 $\diff{}{x} \Phi^{-1}(x) = 1/\phi(\Phi^{-1}(x))$, we have
 \begin{equation}\label{eq:diff_J}
  \partial_i J(x; M)
  = K(\Phi^{-1}(x_{\hat i}) - M_{i\hat i} \Phi^{-1}(x_i); \overline{M_i})
 \end{equation}
 (where by $\Phi^{-1}(x_{\hat i})$, we mean the vector
 obtained by applying $\Phi^{-1}$ to $x_{\hat i}$ element-wise.)
 Now let $I(x) = \phi(\Phi^{-1}(x))$ and define, for $j \ne i$,
 \[
  J_{ij}(x; M) = I(x_i) \partial_j K(\Phi^{-1}(x_{\hat i}) - M_{i\hat i} \Phi^{-1}(x_i); \overline{M_i});
 \]
 by the chain rule applied to~\eqref{eq:diff_J}, we have
 \begin{align}
  \partial_i \partial_j J(x; M) 
  &= \frac{1}{I(x_j)}
  \partial_j K(\Phi^{-1}(x_{\hat i}) - M_{i\hat i} \Phi^{-1}(x_i); \overline{M_i}) \notag \\
  &= \frac{1}{I(x_i) I(x_j)} J_{ij}(x; M).
  \label{eq:2diff-J-mixed}
 \end{align}
 It is worth mentioning that this last equation shows that in fact
 $J_{ij} = J_{ji}$. This is not obvious from the definition of $J_{ij}$, although
 it may also be checked by the tedious process of calculating the derivative
 in that definition.
 
 To compute the repeated second derivatives of $J$, we use~\eqref{eq:diff_J}
 and the chain rule to write
 \begin{align}
  \partial_i^2 J(x; M)
  &= -\sum_{j \ne i} \frac{m_{ij}}{I(x_i)} \partial_j K(\Phi^{-1}(x_{\hat i}) - M_{i\hat i} \Phi^{-1}(x_i); \overline{M_i}) \notag \\
  &= -\frac{1}{I^2(x_i)} \sum_{j \ne i} m_{ij} J_{ij}(x; M).
  \label{eq:2diff-J-diag}
 \end{align}
 Now let $\I(x)$ be the $k \times k$ matrix with $1/I(x_i)$ as the $i$th
 diagonal entry.
 Note that by~\eqref{eq:2diff-J-mixed}, the $ij$ entry of
 $M \odot H_J$ is given by $\frac{m_{ij}}{I(x_i) I(x_j)} J_{ij}$,
 while the $ii$ entry of $M \odot H_J$ is just given
 by~\eqref{eq:2diff-J-diag} (since $m_{ii} = 1$).
 Hence we may write
 \begin{equation}\label{eq:hess-J}
  M \odot H_J(x; M) = \I(x) A(x) \I(x),
 \end{equation}
 where $a_{ij} = m_{ij} J_{ij}$ and $a_{ii} = - \sum_{j \ne i} a_{ij}$.
 
 \begin{lemma}\label{lem:neg-def}
  If $A$ is a matrix such that $a_{ij} \ge 0$ for $i \ne j$ and
  $a_{ii} = -\sum_{j \ne i} a_{ij}$ then $A \le 0$.
 \end{lemma}
 
 \begin{proof}
  In fact, the proof follows immediately from some well-known facts in linear
  algebra, such as the fact that $-A$ is diagonally dominant. However,
  we may also give a simple proof by noting that the quadratic form of $A$ is
  nothing but
  \[
   v^T A v = - \sum_{i<j} a_{ij} (v_i - v_j)^2 \le 0.
  \]
  This formula also leads us to the observation that if all of the $a_{ij}$
  are strictly positive (which is the case for the matrix $A$
  in~\eqref{eq:hess-J}) then the kernel of $A$ is the span of the all-ones
  vector. This observation is irrelevant to Lemma~\ref{lem:neg-def}, but
  we will use it later.
 \end{proof}

 Combining Lemma~\ref{lem:neg-def} with~\eqref{eq:hess-J}, we have proved
 Proposition~\ref{prop:hess-J}, and hence proved~\eqref{eq:main}.
 
 To complete the proof of Theorem~\ref{thm:main}, it remains to characterize
 the equality cases; for this, we will use Corollary~\ref{cor:linear-diffusion}:
 if $\E J(F) = J(\E F)$ then for every $t > 0$,
  \begin{align}
   0 &= \big((M \odot H_\Psi(P_t F^Q)) \otimes I_n\big)
   \begin{pmatrix} (\grad P_t f_1)^{Q_1}\\ \vdots \\ (\grad P_t f_k)^{Q_k} \end{pmatrix} \notag \\
   &= \big((\I(P_t F^Q) A(P_t F^Q) \I(P_t F^Q)) \otimes I_n\big)
   \begin{pmatrix} (\grad P_t f_1)^{Q_1}\\ \vdots \\ (\grad P_t f_k)^{Q_k} \end{pmatrix} \notag \\
   &= (\I(P_t F^Q)\otimes I_n) ( A(P_t F^Q)\otimes I_n)( \I(P_t F^Q)) \otimes I_n)
   \begin{pmatrix} (\grad P_t f_1)^{Q_1}\\ \vdots \\ (\grad P_t f_k)^{Q_k} \end{pmatrix},
   \label{eq:equality-1}
  \end{align}
  where the second equality follows from~\eqref{eq:hess-J}. Since
  $\I$ is always non-singular, we may drop the first instance of it
  from~\eqref{eq:equality-1}. Defining $w_i = \Phi^{-1} \circ f_i$, we have
  $\grad w_i = \grad f_i / I(f_i)$. By multiplying out the last two terms
  of~\ref{eq:equality-1}, we have
  \[
   0 = ( A(P_t F^Q)\otimes I_n)
   \begin{pmatrix} \frac{(\grad P_t f_1)^{Q_1}}{I(P_t f_1)^{Q_1}} \\ \vdots \\ \frac{(\grad P_t f_k)^{Q_k}}{I(P_t f_k)^{Q_k}} \end{pmatrix}
   = ( A(P_t F^Q)\otimes I_n)
    \begin{pmatrix} (\grad w_1)^{Q_1} \\ \vdots \\ (\grad w_k)^{Q_k} \end{pmatrix}
  \]
  Now, recall from the proof of Lemma~\ref{lem:neg-def}
  that the kernel of $A$ is $(1 \dots 1)^T$. It follows then that if
  \[
   (A \otimes I_n)
    \begin{pmatrix} (\grad w_1)^{Q_1} \\ \vdots \\ (\grad w_k)^{Q_k} \end{pmatrix}
    = 0
  \]
  then the $(\grad w_i)^{Q_i}$ are all equal. Since this holds pointwise, and
  since the distribution of $Z$ is fully supported on $\R^{kn}$, we see
  that $\grad w_i$ must all be almost surely equal to the same constant, $a$ say.
  Since each $w_i$ is a smooth function, we have
  $w_i(x) = a \cdot x + b_i$ for some $b_i$. Recalling the definition of $w_i$,
  we have $(P_t f_i)(x) = \Phi(a \cdot x + b_i)$. Carlen and
  Kerce~\cite{CarlenKerce:01} observed
  (and this observation was subsequently used in~\cite{MosselNeeman:12a} and~\cite{MosselNeeman:12b})
  that under this condition, and if $f_i = 1_{A_i}$, then $A_i$
  is a half-space and $a$ is normal to it. Since we have the same $a$ for every
  $A_i$, it follows that $A_1, \dots, A_k$ is a family of parallel half-spaces,
  which completes the proof of Theorem~\ref{thm:main}.
  
  In order to be more self-contained, let us sketch a proof
  (from~\cite{MosselNeeman:12b}) of why $(P_t 1_A)(x) = \Phi(a \cdot x + b_i)$ implies
  that $A$ is a half-space. First, one checks that if $A$ is a half-space
  and $\nu$ its outward unit normal, then
  \[
    (P_t 1_A)(x) = \Phi(k_t \nu \cdot x + b)
  \]
  for some $b \in \R$, where $k_t = (e^{2t} - 1)^{-1/2}$. Moreover, as
  $A$ ranges over all half-spaces with normal $\nu$ then $b$ ranges over $\R$.
  After checking that $P_t$ is one-to-one, this implies that if
  $(P_t 1_A)(x) = \Phi(a \cdot x + b)$ with $|a| = k_t$, then $A$ is a half-space
  with normal $a$. It remains to see what happens when $|a| \ne k_t$. First of
  all, Bakry and Ledoux showed that $|\grad (\Phi^{-1} \circ P_t f)| \le k_t$
  for any $f: \R^n \to [0, 1]$; hence $|a| \le k_t$. But if $|a| < k_t$ then
  there is some $s > 0$ with $|a| = k_{t + s}$. It follows from the previous
  argument, then, that there is a half-space $B$ with
  $P_{t+s} 1_B = \Phi(a \cdot x + b) = P_t 1_A$. We then have
  $P_s 1_B = 1_A$, which is a contradiction since $P_s 1_B$ is always a smooth
  function.

 \bibliographystyle{plain}
 \bibliography{all,../elchanan}

\end{document}